\documentclass[11pt,reqno]{amsart}
 \usepackage[dvips]{epsfig}
 \usepackage{amsgen, amstext,amsbsy,amsopn, amsthm, amsfonts,amssymb,amscd,amsmat
 h,euscript,enumerate,url,verbatim,calc,xy}

 \usepackage{latexsym}
 \usepackage{graphics}
 \usepackage{color}

 \newcommand{\m}{\mathfrak{m} }

 \newcommand{\ov}{\overline}

  \newcommand{\im}{\operatorname{image}}

 \newcommand{\depth}{\operatorname{depth}}

  \newcommand{\lm}{\lambda}

\theoremstyle{plain}
 \newtheorem{theorem}{Theorem}[section]
 
 \newtheorem{corollary}[theorem]{Corollary}
 \newtheorem{lemma}[theorem]{Lemma}
 \newtheorem{proposition}[theorem]{Proposition}

 \theoremstyle{definition}

 \theoremstyle{remark}

\title[ ] {On the Chern number of  $I$-admissible filtrations of ideals}

\author{Mousumi Mandal and J. K. Verma}

\address{Dipartimento Di Matematica, Universita' Di Genova, Via Dodecaneso 35, 16146, Genova, Italy} \email{mandal@dima.unige.it}

\address{Department of Mathematics, Indian Institute of Technology 
Bombay, Powai, Mumbai 400076, India} \email{jkv@math.iitb.ac.in}

\begin{document}
\maketitle
\begin{abstract}
Let $I$ be an $\m$-primary ideal of a Noetherian local ring $(R, \m)$ of positive dimension. The coefficient $e_1(\mathcal I)$ of the Hilbert polynomial of an $I$-admissible filtration $\mathcal I$ is called the Chern number of $\mathcal I$. A formula  for the  Chern number has been derived involving Euler characteristic of subcomplexes of  a Koszul complex. Specific formulas for the Chern number have been given in local rings of dimension at most two. These have been used to provide new and unified proofs of several results about $e_1(\mathcal I)$.
\end{abstract}
\thispagestyle{empty}
\section*{Introduction}
 Let $(R,\m)$ be a Noetherian local ring of dimension $d$ and let $I$ be an $\m$-primary ideal. A sequence of ideals $\mathcal I=\{I_n\}_{n\in \mathbb Z}$ is called an 
 $I$-{\bf admissible filtration} if there exists a $k\in \mathbb N$ such that for all $n,m\in \mathbb Z,$  
$$ I_{n+1}\subseteq I_n, \; I_m I_n\subseteq I_{n+m}\;\; \mbox{and}\;\; I^n\subseteq I_n\subseteq I^{n-k}.$$  The Rees algebra  $\mathcal R (\mathcal I)$ and the associated graded ring $G(\mathcal I)$  of the filtration $\mathcal I$ are  defined as $$\mathcal R(\mathcal I)=\bigoplus_{n\in \mathbb Z}I_nt^n\;\; G(\mathcal I)=\bigoplus _{n\geq 0}I_n/I_{n+1}.$$ 
For the $I$-adic filtration $\mathcal I =\{I^n\},$ we put 
$\mathcal R(\mathcal I)=\mathcal R(I)$  and $G(\mathcal I)=G(I).$
Note that  $\mathcal I$ is an $I$-admissible filtration if and only if  $\mathcal R(\mathcal I)$ is a finitely generated $\mathcal R(I)$-module.    Rees  \cite{rees1} proved  that the integral closure 
filtration $\{\ov{I^n}\}$ is an  $I$-admissible filtration if and only if $R$ is analytically unramified. Marley  \cite{m} showed that if $\mathcal I$ is an $I$-admissible filtration then the {\bf Hilbert function} $H_{\mathcal I}(n)=\lm(R/I_n)$, where $\lm$ denotes length as $R$-module, coincides with a polynomial 
$P_{\mathcal I}(x)\in \mathbb Q[x]$ of degree $d$ for large $n$. This polynomial is written as 
 $$P_{\mathcal I}(x)=e_0(\mathcal I){x+d-1\choose d}-e_1(\mathcal{ I}){x+d-2\choose d-1}+\cdots +(-1)^de_d(\mathcal I)$$
and it is called the {\bf Hilbert polynomial} of the $I$-admissible filtration $\mathcal I$ and the uniquely determined integers $e_i(\mathcal I)$ 
for $i=0,\ldots ,d$ are called the {\bf Hilbert coefficients} of $\mathcal I$. The coefficient $e_1(\mathcal I)$ is called the {\bf Chern number} 
associated with $\mathcal I$.  A {\bf reduction} of an $I$-admissible filtration $\mathcal I=\{I_n\}_{n\in \mathbb Z}$ is an ideal $J\subseteq I_1$ such that $JI_n=I_{n+1}$ for all large $n$.
 Equivalently, $J\subseteq I_1$ is a reduction of $\mathcal I$ if and only if $R(\mathcal I)$ is a finite $R(J)$-module. A {\bf minimal reduction} of $\mathcal I$ is a reduction of $\mathcal I$ minimal with respect to containment. Recall that minimal reductions of admissible filtration always exists and generated by $d$ elements if $R/\m$ is infinite. In this paper we  assume that $R/\m$ is infinite.

The Chern number has traditionally been studied in Cohen-Macaulay local rings. The recent solutions by Goto et al 
\cite{gghopv,ghv, ghm} of conjectures of Vasconcelos   \cite{vas} for the Chern numbers of parameter ideals and integral closure filtrations require their understanding in arbitrary local rings. Therefore it is useful
to have versions of important formulas for the Chern numbers
in general.  We have chosen to find a general version of Huneke's fundamental lemma \cite{hun1} for this purpose. This lemma has played a crucial role in various studies of Hilbert polynomial. This version expresses the Chern number in terms of
Euler characteristics of a family of subcomplexes of the Koszul complex.  Unexpectedly, we  are able to apply this formula in dimension one
to derive unified and simpler proofs of several results about Hilbert polynomial in one-dimensional local rings. 
 
The paper  is organized as follows. In the first section  we establish a formula for the Chern number in any dimension  involving the Euler characteristic of subcomplexes of a  Koszul complex. In the second section we derive concrete versions of this formula  in dimension one and  provide unified proofs
 of results of Rees, Sally, Lipman and Huneke. In the third section we unify several results for the Chern number in local rings of dimension  two. Huenke's fundamental lemma is deduced as a consequence.

\section{A Formula for the Chern Number $e_1(\mathcal I)$}
In this section we give a formula for the Chern number of an $I$-admissible filtration in terms of the Euler Characteristics of subcomplexes of a Koszul complex.
Let $(R,\m)$ be a $d$-dimensional local ring and $I$ be an $\m$-primary ideal. Let $\{I_n\}_{n\in \mathbb Z}$ be an $I$-admissible filtration. 
Let $\underbar x= x_1,\ldots ,x_d\in I$ be a minimal reduction of $\mathcal I$. In order to prove the main theorem we define for $n\in \mathbb Z$ the complex $K^{(n)}(\underbar x,\mathcal I)$ which is a sub-complex of the Koszul complex given by
$$0\longrightarrow I_{n-d}\longrightarrow (I_{n-d+1})^{d\choose 1}\longrightarrow \cdots \longrightarrow (I_{n-1})^{d\choose d-1}\longrightarrow I_{n}\longrightarrow 0.$$ 

The Euler characteristic of $K^{(n)}(\underbar
 x,\mathcal I)$ is defined as $$\chi(K^{(n)}(\underbar x,\mathcal I))=\sum_{i=0}^d(-1)^i\lm(H_i(K^{(n)}(\underbar x,\mathcal I))).$$

\noindent
For a numerical function $f: \mathbb Z \rightarrow \mathbb Z$ we put 
$ (\bigtriangleup f)(n)=f(n)-f(n-1).$
\begin{theorem}\label{cor1}
Let $(R,\m)$ be a $d$-dimensional local ring. Let $I$  be an $\m$-primary ideal and $\mathcal I=\{I_n\}_{n\in \mathbb Z}$ be an $I$-admissible filtration. Let $\underbar x=(x_1,\ldots, x_d)\subseteq I$ be a minimal reduction of $\mathcal I$. Then
$$e_1(\mathcal I)=\sum_{n=1}^{\infty}\chi(K^{(n)}(\underbar x,\mathcal I)).$$

\end{theorem}

\begin{proof}
Consider for $n\in \mathbb Z$ the modified Koszul complex $C_.(n,\mathcal I)$ introduced in \cite{hm} as
 \begin{equation*}
 0\longrightarrow R/I_{n-d}\longrightarrow \left( R/I_{n-d+1}\right)^{d \choose 1}\longrightarrow \cdots \longrightarrow \left( R/I_{n}\right)\longrightarrow 0.
\end{equation*}
with $I_n=R$ for $n\leq 0$ and the differentials are induced by the differentials of the Koszul complex $K_.=K_.(\underbar x)$.
 Then we get
{\small \begin{equation*}
  H_{\mathcal I}\left( n\right) -{d\choose 1}H_{\mathcal I}\left(n-1\right) +\cdots +(-1)^dH_{\mathcal I}\left( n-d\right) =\displaystyle{\sum_{i=0}^d(-1)^{i}\lm(H_i(C_.(n,\mathcal I)))}
 \end{equation*}}
Notice that $\bigtriangleup ^d P_{\mathcal I}(n)=e_0(\mathcal I)$ and 
 $$\bigtriangleup ^d H_{\mathcal I}\left( n\right) =H_{\mathcal I}\left( n\right) -{d\choose 1}H_{\mathcal I}\left(n-1\right) +\cdots +(-1)^dH_{\mathcal I}\left( n-d\right) .$$
 Therefore $$\bigtriangleup^d\left[ P_{\mathcal I}\left( n\right) -H_{\mathcal I}\left(n\right) \right] =e_0\left( \mathcal I \right) -\chi(C_.(n,\mathcal I)).$$
 By Serre's Theorem \cite[Theorem 4.7.6]{bh} we have $e_0(\mathcal I)=\chi(K_.)$. Therefore  $$\bigtriangleup^d\left[ P_{\mathcal I}\left( n\right) -H_{\mathcal I}\left(n\right) \right] =\chi(K_.) -\chi(C_.(n,\mathcal I)).$$
  Consider the exact sequence of complexes 
 \begin{equation}\label{eq11}
 0\longrightarrow K^{(n)}(\underbar x,\mathcal I)\longrightarrow K_.(\underbar x)\longrightarrow C_.(n,\mathcal I)\longrightarrow 0
 \end{equation}
Then from the exact sequence (\ref{eq11}) we have $$\chi(K_.)=\chi(K^{(n)}(\underbar x,\mathcal I))+\chi(C_.(n,\mathcal I)).$$
Therefore  
\begin{equation}\label{lm}
\bigtriangleup^d\left[ P_{\mathcal I}\left( n\right) -H_{\mathcal I}\left(n\right) \right] =\chi(K^{(n)}(\underbar x, \mathcal I)).
\end{equation}
From Proposition 2.9 of \cite{huc} it follows that 
$$e_1(\mathcal I)=\displaystyle{\sum_{n=1}^{\infty}\bigtriangleup^d[P_{\mathcal I}(n)-H_{\mathcal I}(n)]}=\sum_{n=1}^{\infty}\chi(K^{(n)}(\underbar x,\mathcal I)).$$
\end{proof}

\section{Applications  of the Formula in dimension one}
In this section we explicitly write down the formula for the Chern number in one dimensional local rings as consequences of Theorem \ref{cor1}. As applications we unify several results in dimension one.  The next result was proved for Cohen-Macaulay local rings in \cite{hm}.
\begin{theorem}\label{1thm}
 Let $(R,\m)$ be a $1$-dimensional local ring and $I$ be an $\m$-primary ideal and $\mathcal I=\{I_n\}_{n\in \mathbb Z}$ be an $I$-admissible filtration. Let $J=(x)\subseteq I_1$ be a minimal reduction. Then
 $$e_1(\mathcal I)=\sum_{n=1}^{\infty}[\lm(I_{n}/JI_{n-1})-\lm((0:x)\cap I_{n-1})].$$
 In particular, if $R$ is Cohen-Macaulay then  $$e_1(\mathcal I)=\sum_{n=1}^{\infty}\lm(I_{n}/JI_{n-1}).$$
\end{theorem}

\begin{proof}
 From Theorem \ref{cor1} we have
 $$e_1(\mathcal I)=\sum_{n=1}^{\infty}\chi(K^{(n)}(\mathcal I))$$
 where the complex $K^{(n)}(\mathcal I)$ is given as 
 $$0\longrightarrow I_{n-1}\xrightarrow{x} I_{n}\longrightarrow 0.$$
 Note that $$H_0(K^{(n)}(\mathcal I))=I_{n}/JI_{n-1}$$ and $$H_1((K^{(n)}(\mathcal I))=(0:x)\cap I_{n-1}.$$
 Hence $$e_1(\mathcal I)=\sum_{n=1}^{\infty}[\lm(I_{n}/JI_{n-1})-\lm((0:x)\cap I_{n-1})].$$
  If $R$ is Cohen-Macaulay, then $x$ is a regular element and so $(0:x)=0$. Hence we have
  $$e_1(\mathcal I)=\sum_{n=1}^{\infty}\lm(I_{n}/JI_{n-1}).$$
\end{proof}

The next result is a deep theorem of Rees valid in any dimension.  The proof given here in dimension one for admissible filtrations is new and straightforward.
\begin{theorem}{\rm \cite[Rees]{rees}}
Let $(R,\m)$ be a $1$-dimensional local ring. Let $I$ be an $\m$-primary ideal and $\mathcal I=\{I_n\}_{n\in \mathbb Z}$ be an $I$-admissible filtration. Let $J\subseteq I_1$ be a parameter ideal such that $e_0(J)=e_0(\mathcal I)$. Then $J$ is a reduction of $\mathcal I$.
\end{theorem}
\begin{proof}
 Let $J=(x)\subseteq I_1$.
 Consider the complex
 $$C_.(n,J,\mathcal I):~~~0\longrightarrow I_{n-1}/J^{n-1} \xrightarrow{x} I_{n}/J^{n}\longrightarrow 0.$$
 Let $H_{\mathcal I/J}(n)=\lm(I_n/J^n)$ be the Hilbert function and $P_{\mathcal I/J}(n)$ denote the corresponding Hilbert polynomial. 
 Note that 
 $$\bigtriangleup[P_{\mathcal I/J}(n)-H_{\mathcal I/J}(n)]=e_0(\mathcal I/J)-\lm(H_0(C_.(n,J,\mathcal I)))+\lm(H_1(C_.(n,J,\mathcal I))).$$
 Since $e_0(\mathcal I)=e_0(J)$, $e_0(\mathcal I/J)=0$. Hence for large $n$ we have
 $$\lm(H_1(C_.(n,J,\mathcal I)))-\lm(H_0(C_.(n,J,\mathcal I)))=0.$$
 Notice that $$H_0(C_.(n,J,\mathcal I))=\dfrac{I_{n}}{JI_{n-1}}$$ and $$H_1(C_.(n,J,\mathcal I))=\dfrac{(J^{n}:x)\cap I_{n-1}}{J^{n-1}}.$$
 Observe that $(J^{n}:x)=J^{n-1}+(0:x)$. Hence we have
 \begin{eqnarray*}
  \dfrac{(J^{n}:x)\cap I_{n-1}}{J^{n-1}}&=& \dfrac{J^{n-1}+(0:x)\cap I_{n-1}}{J^{n-1}}\\
  &=& \dfrac{(0:x)\cap I_{n-1}}{(0:x)\cap J^{n-1}}
 \end{eqnarray*}
By Artin-Rees Lemma, $$(0:x)\cap I_n\subseteq H^0_{\m}(R)\cap I_n=I_{n-n_0}(I_{n_0}\cap H^0_{\m}(R))$$ for some $n_0$ and for all $n\geq n_0$. 
Since $\lm(H^0_{\m}(R))<\infty$, for large $n$ we have $$I_{n-n_0}(I_{n_0}\cap H^0_{\m}(R))=0.$$ Hence for large $n$, $I_{n}=JI_{n-1}$. Thus $J$ is a reduction of $\mathcal I$.
\end{proof}

\begin{theorem}{\rm \cite[Theorem 1.9]{l}}
 Let $(R,\m)$ be a $1$-dimensional Cohen-Macaulay ring and $I$ be an $\m$-primary ideal and $\mathcal I=\{I_n\}_{n\in \mathbb Z}$ be an 
$I$- admissible filtration. Let $J=(x)$ be a minimal reduction of $\mathcal I$. Then for all $n\geq 1$
 $$\lm(I_{n-1}/I_{n})\leq e_0(\mathcal I) $$ and
 $$\lm(I_{n-1}/I_{n})= e_0(\mathcal I) \mbox { if and only if } I_{n}=xI_{n-1}.$$ 
\end{theorem}

\begin{proof}
 From Equation (\ref{lm}) we have
 $$\bigtriangleup[P_{\mathcal I}(n)-H_{\mathcal I}(n)]=\lm(I_{n}/JI_{n-1})$$
 which implies $$e_0(\mathcal I)-\lm(R/I_{n})+\lm(R/I_{n-1})=\lm(I_{n}/JI_{n-1}).$$
 Thus we have $\lm(I_{n-1}/I_{n})+\lm(I_{n}/JI_{n-1})=e_0(\mathcal I)$. Hence $\lm(I_{n-1}/I_{n})\leq e_0(\mathcal I)$ for $n\geq 1$ and equality holds if and only if $I_{n}=JI_{n-1}$. 
 
\end{proof}

\begin{theorem}{\rm \cite[Theorem 2.1]{hun1}}
  Let $(R,\m)$ be a $1$-dimensional Cohen-Macaulay ring and $I$ be an $\m$-primary ideal and $\mathcal I=\{I_n\}_{n\in \mathbb Z}$ be an $I$-admissible filtration. Assume that $$e_1(\mathcal I)=e_0(\mathcal I)-\lm(R/I_1).$$ Then
 \begin{enumerate}
  \item[{\rm (1)}] $H_{\mathcal I}(n)=P_{\mathcal I}(n)$ for all $n\geq 1$.
  \item[{\rm (2)}] If $(x)$ is a minimal reduction of $\mathcal I$ then $I_n=xI_{n-1}$ for all $n\geq 2$.
 \end{enumerate}
\end{theorem}

\begin{proof}
 From Theorem \ref{1thm} we have $e_1(\mathcal I)=\displaystyle{\sum_{n\geq 1}}\lm\left( I_{n}/JI_{n-1}\right)$. Since $$e_1(\mathcal I)=
e_0(\mathcal I)-\lm(R/I_1)=\lm(R/J)-\lm(R/I_1)=\lm(I_1/J),$$ we get $\displaystyle{\sum_{n\geq 2}}\lm\left( I_{n}/JI_{n-1}\right)=0.$ Hence $I_n=JI_{n-1}$
for all $n\geq 2$.
 From  Equation (\ref{lm}) we have $$\bigtriangleup [P_{\mathcal I}(n)-H_{\mathcal I}(n)]=0$$ for $n\geq 2$. Hence $$P_{\mathcal I}(n)-H_{\mathcal I}(n)=
P_{\mathcal I}(n-1)-H_{\mathcal I}(n-1).$$ Notice that $$P_{\mathcal I}(1)-H_{\mathcal I}(1)=e_0(\mathcal I)-e_1(\mathcal I)-\lm(R/I_1)=0.$$
 Hence $H_{\mathcal I}(n)=P_{\mathcal I}(n)$ for all $n\geq 1$.
\end{proof}

\begin{theorem}{\rm \cite[Proposition 3.2]{s}}
 Let $(R,\m)$ be a $1$-dimensional Cohen-Macaulay ring and $I$ be an $\m$-primary ideal. Assume that $$e_1(I)=e_0(I)-\lm(R/I)+1.$$ Then
 \begin{enumerate}
  \item[{\rm (1)}] $H_{I}(n)=P_{I}(n)$ for all $n> 1$.
  \item[{\rm (2)}] If $(x)$ is a minimal reduction of $I$ then $\lm(I^2/xI)=1$ and $I^3=xI^2$.
 \end{enumerate}

\end{theorem}

\begin{proof}
 From Theorem \ref{1thm} we have $e_1(I)=\displaystyle{\sum_{n=1}^{\infty}}\lm\left( \dfrac{I^{n}}{JI^{n-1}}\right) $. Since $$e_1(I)=e_0(I)-\lm(R/I)+1=\lm(R/J)-\lm(R/I)+1=\lm(I/J)+1$$ we get
 $$\displaystyle{\sum_{n\geq 2}}\lm(I^n/JI^{n-1})=1 ~~\mbox{ and } ~~\lm(I^2/JI)=1 ~~\mbox{ and }~~ I^3=JI^2.$$
 From Equation (\ref{lm}) we have $\bigtriangleup [P_I(n)-H_I(n)]=0$ for all $n\geq 3$ and $\lm(I^2/xI)=1$. Note that $$P_I(n)-H_I(n)=P_I(n-1)-H_I(n-1)$$ for all $n\geq 3$ and $$P_I(2)-H_I(2)=P_I(1)-H_I(1)+1=e_0(I)-e_1(I)-\lm(R/I)+1=0$$
 Hence $H_I(n)=P_I(n)$ for all $n>1$.
\end{proof}

\section{Applications of the Formula in dimension two}
In this section we derive the formula for the Chern number in dimension 2 as a consequence of Theorem \ref{cor1} and as an application derive Huneke's fundamental Lemma.  

\begin{proposition}\label{pro1}
 Let $(R,\m)$ be a $2$-dimensional local ring and let $J=(x,y)$ be a parameter ideal. Let $\mathcal J=\{J_n\}$ be a $J$-admissible filtration  with $\depth G(\mathcal J)\geq 1$ and $x^*\in J_1/J_2$ be a nonzerodivisor in $G (\mathcal J)$. Let $K_.^{(n)}(\mathcal J)$ be the subcomplex of the Koszul complex $K.(x,y)$  defined as :
 $$0\longrightarrow J_{n-2} \xrightarrow{f_n=(y,x)} (J_{n-1})^2 \xrightarrow{g_n=\left( \begin{array}{r}   -x \\ y  \end{array} \right) } J_{n}\longrightarrow 0 .$$
 Then $$H_1(K_.^{(n)}(\mathcal J))\cong\frac{(x:y)\cap J_{n-1}}{(x)J_{n-2}}.$$
\end{proposition}
\begin{proof}
 In order to calculate the first homology of the above complex we calculate $\ker g_n$ and $\im f_n$. 
 \begin{eqnarray*}
  \ker g_n &=& \{(a,b)\in (J_{n-1})^2\mid -ax+by=0\}\\
           &=& \{(a,b)\in (J_{n-1})^2\mid ax=by\}
 \end{eqnarray*}
 Observe that $\ker g_n\cong (x:y)\cap J_{n-1}$ through the map $(a,b)\mapsto b$.
 Indeed, as  $x$ is a regular element, for each $b\in(x:y)$ there exists a unique $a\in R$ such that $ax=by$ and since $x^*$ is $G(\mathcal J)$- 
regular $J_{n}:x=J_{n-1}$ for all $n\geq 2$, so $a\in J_{n-1}$. Thus $\ker g_n\cong (x:y)\cap J_{n-1}$.
 Notice that $\im f_n=\{(cy,cx)|c\in J_{n-2}\}$. Since $\im f_n \subseteq \ker g_n$, we have $\im f_n\cong (x)J_{n-2}.$
Hence $$H_1(K_.^{(n)}(\mathcal J))\cong\dfrac{(x:y)\cap J_{n-1}}{(x)J_{n-2}}.$$ 
\end{proof}

\begin{theorem}\label{dm2}
 Let $(R,\m)$ be a $2$-dimensional local ring and $J=(x,y)$ be a parameter ideal. Let $\mathcal J=\{J_n\}$ be a $J$-admissible filtration and 
$\depth G(\mathcal J)\geq 1$ and $x^*\in J_1/J_2$ be a nonzerodivisor in $G(\mathcal J)$. Then
 $$ e_1(\mathcal J)=\sum_{n\geq 1}\left[ \lm\left( \dfrac{{J_{n}}}{J{J_{n-1}}}\right) -\lm\left( \dfrac{(x:y)\cap {J_{n-1}}}{(x){J_{n-2}}}\right)\right] $$
\end{theorem}

\begin{proof}

By Theorem \ref{cor1} we have
$$e_1(\mathcal J)=\sum_{n\geq 1}\chi(K^{(n)}(\mathcal J)).$$
Notice that $$H_0(K^{(n)}(\mathcal J))=J_{n}/JJ_{n-1}$$ and $$H_2(K^{(n)}(\mathcal J))=(0:J)\cap J_{n-2}.$$
Since $\depth G(\mathcal J)\geq 1$ and $x^*$ is a nonzerodivisor so $x$ is a nonzerodivisor  and hence $H_2(K_.)=0.$ Therefore $H_2(K^{(n)}(\mathcal J))=0$
for all $n$. By Proposition \ref{pro1} we have
$$H_1(K_.^{(n)}(\mathcal J))\cong\frac{(x:y)\cap {J_{n-1}}}{(x)J_{n-2}}.$$ Thus 
$$ e_1(\mathcal J)=\sum_{n\geq 1}\left[ \lm\left( \dfrac{{J_{n}}}{J{J_{n-1}}}\right) -\lm\left( \dfrac{(x:y)\cap {J_{n-1}}}{(x){J_{n-2}}}\right)\right] .$$
\end{proof}

\begin{corollary}\label{ghun}
 Let $(R,\m)$ be a $2$-dimensional analytically unramified local ring and $J=(x,y)$ be a parameter ideal. Then
 $$\ov e_1(J)=\sum_{n\geq 1}\left[ \lm\left( \dfrac{\ov{J^{n}}}{J\ov{J^{n-1}}}\right) -\lm\left( \dfrac{(x:y)\cap \ov{J^{n-1}}}{(x)\ov{J^{n-2}}}\right)\right] $$
 where $\ov J$ denote the integral closure of $J$.
\end{corollary}

\begin{proof}
 Since $R$ is analytically unramified, $\mathcal J=\{\ov{J^n}\}$ is a $J$-admissible filtration and by \cite[Theorem 3.25]{m} $\depth G(\mathcal F)\geq 1$. Hence by Theorem \ref{dm2} we are done.
\end{proof}

Next we derive Huneke's Fundamental Lemma for $I$-admissible filtrations. For this purpose we need the following Lemma.
\begin{lemma}\label{l3}{\rm\cite[Lemma 3.2]{hm}}
Let $(R,\m)$ be a $d$-dimensional local ring. Let $I$ be an $\m$-primary ideal and $\mathcal I=\{I_n\}_{n\in \mathbb Z}$ be an $I$-admissible filtration. Let $x_1,\ldots ,x_d$ be elements of $I_1$. Consider the complex $C_.(n,\mathcal I)$ 
$$0\longrightarrow R/I_{n-d}\longrightarrow (R/I_{n-d+1})^{d\choose 1}\longrightarrow \cdots \longrightarrow R/I_{n}\longrightarrow 0.$$
Then for $n\in \mathbb Z$, the following hold:
 \begin{enumerate}
  \item[{\rm (1)}] $H_d(C_.(n,\mathcal I))=\dfrac{(I_{n-d+1}:(x_1,\ldots ,x_d))}{I_{n-d}}.$
 \item[{\rm  (2)}] $H_0(C_.(n,\mathcal I))=R/(I_{n}+(x_1,\ldots ,x_d)).$
  \item[{\rm (3)}] if $x_1,\ldots ,x_d$ is an $R$-regular sequence then
   $$H_{1}(C_.(n,\mathcal I))\cong \dfrac{(x_1,\ldots ,x_d)\cap I_{n}}{(x_1,\ldots ,x_d)I_{n-1}}.$$
 \end{enumerate}
\end{lemma}

 \begin{proposition}{\rm \cite[Fundamental Lemma 2.4]{hun1}}
  Let $(R,\m)$ be a $2$-dimensional Cohen-Macaulay local ring and $I$ be an $\m$-primary ideal and $\mathcal I=\{I_n\}_{n\in \mathbb Z}$ be an 
 $I$-admissible filtration. Let $J=(x,y)$ be a minimal reduction of $\mathcal I$. Then for all $n\geq 2$ 
  \begin{enumerate}
 \item[{\rm (1)}] $\bigtriangleup^2[P_{\mathcal I}(n)-H_{\mathcal I}(n)]=\lm\left( \dfrac{I_{n}}{JI_{n-1}}\right) -\lm\left( \dfrac{I_{n-1}:J}{I_{n-2}}\right)$.
  \item[{\rm (2)}] $e_1(\mathcal I)=e_0(\mathcal I)-\lm(R/I_1)+\displaystyle{\sum_{n\geq 2}}\left[ \lm\left( \dfrac{I_{n}}{JI_{n-1}}\right) -\lm\left( \dfrac{I_{n-1}:J}{I_{n-2}}\right)\right] .$ 
  \end{enumerate}
 \end{proposition}

\begin{proof}
 From  Equation (\ref{lm}) we have
  $$\bigtriangleup^2[P_{\mathcal I}(n)-H_{\mathcal I}(n)]=\chi(K^{(n)}(\mathcal I))$$
  where $K^{(n)}(\mathcal I)$ is the  complex
  $$0\longrightarrow I_{n-2}\longrightarrow (I_{n-1})^2\longrightarrow I_{n}\longrightarrow 0.$$
  Note that $H_0(K^{(n)}(\mathcal I))=I_{n}/JI_{n-1}$ and $H_2(K^{(n)}(\mathcal I))=(0:J)\cap I_{n-2}=0$ as $R$ is Cohen-Macaulay. From the exact sequence of 
complexes (\ref{eq11}) we get the long exact sequence 
  \begin{eqnarray*}
  0 &\longrightarrow H_2(K^{(n)}(\mathcal I))\longrightarrow H_2(K_.)\longrightarrow H_2(C_.(n,\mathcal I))\longrightarrow\\
    & \longrightarrow H_1(K^{(n)}(\mathcal I))\longrightarrow H_1(K_.)\longrightarrow H_1(C_.(n,\mathcal I))\longrightarrow
    \end{eqnarray*}
 Since $R$ is Cohen-Macaulay $H_1(K_.)=H_2(K_.)=0$. Thus $$H_1(K^{(n)}(\mathcal I))\cong H_2(C_.(n,\mathcal I))=\frac{I_{n-1}:J}{I_{n-2}}~~~~~\mbox{ (by Lemma \ref{l3}) }$$ 
 Hence $$\bigtriangleup^2[P_{\mathcal I}(n)-H_{\mathcal I}(n)]=\lm\left( \dfrac{I_{n}}{JI_{n-1}}\right) -\lm\left( \dfrac{I_{n-1}:J}{I_{n-2}}\right).$$
Now from Theorem \ref{cor1} we have
\begin{eqnarray*}
e_1(\mathcal I)= e_0(\mathcal I)-\lm(R/I_1)+\sum_{n\geq 2}\left[ \lm\left( \dfrac{I_{n}}{JI_{n-1}}\right) -\lm\left( \dfrac{I_{n-1}:J}{I_{n-2}}\right)\right] .
\end{eqnarray*}

 \end{proof}

\end{document}